\documentclass{article}
\usepackage[utf8]{inputenc}

\usepackage[a4paper]{geometry}
\newcommand{\nat}{\ensuremath {\mathbb N} }

\newcommand{\Prob}{\mathbb{P}}
\newcommand{\eps}{\varepsilon}
\newcommand{\E}{\mathbb E}

\newcommand{\bin}{\mathrm{Bin}}
\newcommand{\G}{\mathbb{G}}
\newcommand{\Gp}{\mathcal{G}}
\newcommand{\Pk}{\mathbb{P}_k}
\newcommand{\Gk}{\mathbb{\hat{\G}}_k}
\newcommand{\ccc}{\mathcal C}
\newcommand{\aaa}{\mathcal A}

\usepackage[]{natbib}
\setcitestyle{aysep={}}

\usepackage{graphicx}
\usepackage[colorlinks=true, allcolors=blue]{hyperref}

\usepackage{paralist}
\usepackage{amsmath}
\usepackage{amssymb}
\usepackage{amsthm}
\usepackage{algorithmic}
\usepackage{algorithm}
\usepackage{tikz}
\usepackage[]{todonotes}   

\newtheorem{theorem}{Theorem}[section]
\newtheorem{lemma}[theorem]{Lemma}

\newtheorem{algo}[theorem]{Algorithm}
\newtheorem{example}[theorem]{Example}
\newtheorem{remark}[theorem]{Remark}
\newtheorem{observation}[theorem]{Observation}

\author{
Adam Logan\thanks{The Tutte Institute for Mathematics and Computing, Ottawa, ON, Canada.}
\and
Mike Molloy\thanks{Department of Computer Science, University of Toronto, Toronto, ON, Canada.}
\and
Pawe\l{}~Pra\l{}at\thanks{Department of Mathematics, Ryerson University, Toronto, ON, Canada.}
}

\title{A variant of the Erd\H{o}s-R\'enyi random graph process} 

\begin{document}

\maketitle

\begin{abstract}
We consider a natural variant of the Erd\H os-R\'enyi random graph process in which $k$ vertices are special and are never put into the same connected component. The model is natural and interesting on its own, but is actually inspired by the combinatorial data fusion problem that itself is connected to a number of important problems in graph theory. We will show that a phase transition occurs when the number of special vertices is roughly $n^{1/3}$, where $n$ is the number of vertices.
\end{abstract}

\section{Introduction}

The study of the \emph{random graph process} was initiated by Erd\H{o}s and R\'enyi in their celebrated paper from 1959~\cite{ER1959}. The process starts with an empty graph on $n$ vertices and adds all ${n \choose 2}$ edges in a random order. The model is now well understood, though there are still some challenging questions waiting to be answered (see the following monographs on the topic: \cite{bol, JLR,BookFK}). On the other hand, relatively little is known about variants of this process. In particular, a natural variant of the model is the \emph{constrained random graph process} in which, after the edge to be inserted is chosen at random, we check whether the graph at this stage of the process together with this edge satisfies some properties; if so, we accept it, otherwise we reject it and never consider it again. 

The first result on the constrained random graph process is due to Ruci\'nski and Wormald, who answered a question of Erd\H{o}s regarding the process in which we maintain a bound on the maximum degree~\cite{RW_bounded_degree}. Erd\H{o}s, Suen, and Winkler considered both the odd-cycle-free process~\cite{ESW_odd-cycle-free} and the triangle-free process that was later analyzed by Bohman~\cite{Bohman_Triangles}. Other special cases that were considered include the properties of being cycle-free~\cite{A_cycle-free}, $H$-free~\cite{BR_H-free,OT_H-free}, and planarity~\cite{RSA:RSA20186}.

\medskip

In this paper, we consider another natural constrained random graph process in which $k$ vertices are special and never put into the same connected component. This problem was motivated by considering a natural greedy algorithm for the combinatorial data fusion problem.

\medskip

The paper is structured as follows. In Section~\ref{sec:definitions} we introduce necessary definitions and state main results. Connections to the combinatorial data fusion problem are discussed in Section~\ref{sec:motivation}. The random graph process is formally introduced in Section~\ref{sec:more_definitions} where we also make a connection between the two models and list all properties of the original one that we need to understand our model. In Section~\ref{sec:concentration-tools}, we develop some concentration tools that will be used in the proofs. The process shows two different behaviours: $k \ll n^{1/3}$ is considered in Section~\ref{sec:k-small}; Section~\ref{sec:k-large} is devoted to $k \gg n^{1/3} (\log n)^{4/3} (\log \log n)^{1/3}$. Final conclusions are in Section~\ref{sec:remarks}.

\section{Definitions and main results}\label{sec:definitions}

In this section, we introduce the $k$-process we are interested in, the asymptotic notation used throughout the paper, and state the main results. 

\subsection{$k$-process}\label{subsec:k-process}

Let $1 \le k \le n$ be any two integers ($k=k(n)$ may be and usually is a function of $n$). The \emph{$k$-process} starts with $\Pk(n,0)$, the empty graph on $n$ vertices, where $k$ of the vertices are \emph{special}.  For integer $m \ge 1$, create $\Pk(n,m)$ from $\Pk(n,m-1)$ as follows. Choose a random pair of vertices not yet considered (in particular, they are not connected by an edge); connect these two vertices unless doing so would put two of the special vertices in the same component (in which case we say that a \emph{collision} occurs).  Keep repeating these steps, if needed, until one edge is added. In particular, $\Pk(n,m)$ has $m$ edges. The process stops at time $M=M(n,k)$ when $\Pk(n,M)$ has precisely $k$ connected components, each of which is a complete graph. (Of course, $M$ is a random variable counting the number of edges at the end of the process.) Alternatively, one can stop the process much earlier, at time $\hat{M}=\hat{M}(n,k)$ when $\Pk(n,\hat{M})$ has $k$ connected components for the first time, as $M$ and $\Pk(n,M)$ are already determined at this point.

\medskip

The main question raised in this paper is the following one. What can be said about the distribution of sizes of the components of $\Pk(n,M)$? Another natural question is: what can be said about $M(n,k)$ as a function of $k$? What about $\hat{M}(n,k)$?

\subsection{Asymptotics}

As typical in random graph theory, we shall consider only asymptotic properties of $\Pk(n,m)$ (and $\G(n,m)$, $\Gk(n,m)$ defined below) as $n\rightarrow \infty$, where $m=m(n)$ depends on $n$. We emphasize that the notations $o(\cdot)$ and $O(\cdot)$ refer to functions of $n$, not necessarily positive, whose growth is bounded. We use the notations $f \ll g$ for $f=o(g)$ and $f \gg g$ for $g=o(f)$. We also write $f(n) \sim g(n)$ if $f(n)/g(n) \to 1$ as $n \to \infty$ (that is, when $f(n) = (1+o(1)) g(n)$). We say that an event in a probability space holds \emph{asymptotically almost surely} (\emph{a.a.s.}) if its probability tends to one as $n$ goes to infinity.

\subsection{Results}

In this subsection, we summarize the main results proved in this paper. It turns out that the $k$-process changes its behaviour around $k=n^{1/3}$. If $k \ll n^{1/3}$, the giant component is formed before collisions start affecting the process. In particular, when the first special vertex joins the giant its size is much larger than the total size of all other special components. As a result, the giant will continue growing and at the end of the process it will have size $n(1-o(1))$. On the other hand, if $k \gg n^{1/3}$, collisions will start affecting the process much earlier, namely, when each component has size smaller than the total size of all special components. As a result, no component is able to dominate all the others and the largest component in the end has size $o(n)$. For technical reasons (see the final section for a longer discussion), our proofs require slightly larger values of $k$, namely, $k \gg n^{1/3} (\log n)^{4/3} (\log \log n)^{1/3}$.  Below we state these results precisely.

\medskip

For a given graph $G$ (deterministic or random) with $r=r(G)$ connected components, let $L_i(G)$ be the size of an $i$-th largest component ($i=1,2, \ldots, r$). Then the following holds.

\begin{theorem}\label{thm:main}
Let $\omega=\omega(n)$ be any function tending to infinity as $n \to \infty$ (sufficiently slowly). Then a.a.s.\ the following holds:
\begin{itemize}
\item [(a)] If $k \ll n^{1/3}$, then 
$$
L_1(\Pk(n,M)) = n - O(k^3 \omega \log (n^{1/3}/k)) \sim n.
$$
As a result, 
$$ 
M(n,k) \sim { n - k + 1 \choose 2} \sim \frac {n^2}{2}.
$$
\item [(b)] If $k \gg n^{1/3} (\log n)^{4/3} (\log \log n)^{1/3}$ and $k \ll n / \log n$, then 
$$
L_1(\Pk(n,M)) = O \left( \left( \frac {n \log(k/n^{1/3}) \log^4 n}{k^3} \right)^{1/2} n \right) = o(n). 
$$
As a result, 
$$ 
M(n,k) = O \left( \left( \frac {n \log(k/n^{1/3}) \log^4 n}{k^3} \right)^{1/2} n^2 \right) = o(n^2).
$$
\item [(c)] If $n / (\omega \log n) \le k \le n$, then $L_1(\Pk(n,M)) = O (\omega^{3/2} \log^4 n) = o(n)$. As a result, $M(n,k) = O (n \omega^{3/2} \log^4 n) = o(n^2)$.
\end{itemize}
\end{theorem}

Part (a) is proved in Section~\ref{sec:k-small}; part (b) is proved in Section~\ref{sec:k-large}; part (c) follows immediately from part (b) and Observation~\ref{obiggerk} (see Subsection~\ref{sec:part_c} for more details).

\medskip

Note that, trivially and deterministically, we get that $M(n,k) \le { n - k + 1 \choose 2}$; the upper bound holds for the extremal graph on $n$ vertices, $k$ components, and maximum number of edges (union of a complete graph on $n-k+1$ vertices and $k-1$ isolated vertices). Understanding whether $M(n,k)$ is close to this trivial bound has important implications for the applications we consider below.

\section{Motivation}\label{sec:motivation}

Before we formally state the combinatorial data fusion problem, let us start with an important special case. The {\em multiway cut} problem~\cite[Chapter 4]{vazirani} is a standard NP-hard problem in graph theory.  We state it as follows. Let $G$ be a nonnegatively weighted undirected graph with vertex set $V$ and edge set $E$.  Let $S \subseteq V$.  Find the set $R$ of edges of least total weight such that no two vertices in $S$ are in the same connected component of the graph $G_R$ obtained by removing the edges in $R$  from $G$.

There is a natural greedy algorithm that can be used to address this problem:

\begin{algo}{(Edge-first greedy algorithm)}
\begin{enumerate}
\item Begin with an empty graph $G_0$ with the same vertices as $G$.
\item Order the edges of $G$ as $e_1, e_2, \ldots, e_m$ by decreasing weight, breaking ties at random.
\item For each $i$ with $1 \le i \le m$, construct $G_i$ as follows.  If adding $e_i$ to $G_{i-1}$ does not create a connected component containing more than one element of $S$, then $G_i$ is $G_{i-1}$ with $e_i$ added;  otherwise $G_i = G_{i-1}$.
\item Return $G_R=G_m$.
\end{enumerate}
\end{algo}

Clearly this algorithm runs in polynomial time; thus we cannot expect it to find the optimal solution.  In fact, it does not approximate the best solution within any constant factor.

\begin{example}
Let $G = K_{n^2}$, and index the vertices as $v_{i,j}$ for $1 \le i,j \le n$.  Let $S = \{v_{1,j}: 1 \le j \le n\}$, and let the weight of the edge joining $v_{i,j}$ to $v_{k,\ell}$ be $1+\epsilon$ if $j = \ell$ and $1$ otherwise. Note that the total weight is equal to $\binom{n^2}{2} + \binom{n}{2} n\epsilon \sim n^4/2$.  The greedy algorithm adds all edges of weight $1+\epsilon$ and then stops, thus producing a solution of total weight $n \binom{n}{2} (1+\epsilon) \sim n^3 (1+\epsilon) / 2$. On the other hand, let $R$ be the set of all edges incident to at least one vertex of $S \setminus \{v_{1,1}\}$. Then $G_R$ is a solution of total weight $\binom{n^2}{2} + \binom{n}{2} n\epsilon - |R| \sim n^4/2$, since $|R| = \binom{n-1}{2} + (n-1)(n^2-(n-1))+(n-1)^2\epsilon \sim n^3$. This is asymptotically equal to the weight of the whole graph and is larger by a factor of $n$ than the solution found by the greedy algorithm.
\end{example}

Nevertheless, understanding the performance of the greedy algorithm may give some insight into the general problem. As a starting point we considered how it performs on the complete graph.  Let $G = K_n$, and let $S$ be any set of $k$ vertices.   Let the edges either (i) all have weight $1$ or (ii) have i.i.d.\ random weights.   We noted that in either case the greedy algorithm amounts to the $k$-process defined in Section~\ref{subsec:k-process}. 

The main theorem (Theorem \ref{thm:main}) shows that a.a.s.\ for $k \ll n^{1/3}$ we obtain a solution where all but one of the components are very small, while for $k \gg n^{1/3}(\log n)^{4/3}(\log \log n)^{1/3}$ we do not.  It is not hard to see that, in the case where all edges have weight one, the optimal solution has $k-1$ components of size one; so for $k \ll n^{1/3}$ the greedy solution is close to optimal and for larger $k$ it is far from optimal.

\medskip

The multiway cut problem is a special case of the combinatorial data fusion problem which was introduced by Darling et al.~\cite{dhpp}, and can be used to describe many situations that are important in applications.

\medskip

\noindent{\bf The Combinatorial Data Fusion Problem:} Let $G$ be a nonnegatively weighted undirected graph with vertex set $V$ and edge set $E$.  Fix a set $S \subseteq 2^V$ of sets of vertices, called the {\em forbidden sets}.  A set $T \subseteq E$ of edges of $G$ is a solution to the {\em combinatorial data fusion problem} associated to $(G, S)$ if, after removing the edges of $T$, no set of vertices in $S$ is wholly contained in a single connected component. As before, our goal is to find a set $T$ of least total weight.

\medskip

The combinatorial data fusion problem generalizes a number of standard graph-theoretic problems.  For example, suppose that $S$ consists of all $2$-element subsets of a given subset $U \subseteq V$.  Then we have the {\em multiway cut} problem described above.  More generally, if $S$ only contains sets of order $2$, then we have the {\em multicut} problem~\cite[Problem~18.1]{vazirani}. In a rather different direction, let $G$ be a star. Without loss of generality we may assume that no element of $S$ contains the central vertex{: if there is such a set with $2$ elements, then they are joined by an edge, so we delete the edge and the non-central vertex on it and proceed.  If there is such a set with $>2$ elements, it is disconnected if and only if the subset obtained by removing the central vertex is disconnected}.  In this case, disconnecting a set of vertices is the same as removing the edge incident on one of them.  Thus we have the {\em minimum hitting set} problem, which is another standard NP-complete problem~\cite[SP8, p.~222]{gj}:  given a collection of subsets $S$ of a finite set $U$, a {\em hitting set} is a subset of $U$ that meets every element of $S$.  The problem is to determine whether there is a hitting set of size less than $k$.

Considering the more general setting of the combinatorial data fusion problem leads to a broad class of random graph processes:

\medskip

\noindent{\bf The CDF-process:}  Begin with $n$ vertices and specify a collection of {\em forbidden} subsets of those vertices.  Add random edges one at a time by repeatedly choosing a pair of vertices uniformly from all non-edges whose addition to the graph would not form a component containing a forbidden set of vertices.

\medskip

The same questions arise:  how many edges will be added until the process is complete?  What will the components look like? We have not studied this more general process at all, beyond the special case of the $k$-process.

\section{The random graph process and its properties}\label{sec:more_definitions}

In this section, we introduce the random graph process $\G(n,m)$ which will be very useful in the analysis of $\Pk(n,M)$. In particular, a good understanding of the process of forming a giant component in $\G(n,m)$ is needed. We summarize our knowledge on this topic in the last two subsections. 

\subsection{Random graph models}

We fix $n$ vertices. $\G(n,m)$ is the random graph selected uniformly from all graphs on those vertices and with exactly $m$ edges. Equivalently, we can select $\G(n,m)$ by the following process:

\medskip

\noindent{\bf Random graph process}: $e_1, \ldots, e_{\binom{n}{2}}$ is a sequence of pairs of vertices selected uniformly without repetition.  $\G(n,m)$ is the graph formed by edges $e_1,\ldots,e_m$.

\medskip

We can couple this to a process that is essentially identical to $\Pk(n,m)$, differing only in what $m$ counts:

\medskip

\noindent{\bf $\Gk(n,m)$-process}: $S$ is a set of $k$ {\em special} vertices from amongst our $n$ vertices. As before, $e_1, \ldots, e_{\binom{n}{2}}$ is a sequence of pairs of vertices selected uniformly without repetition. $\Gk(n,m)$ is the graph formed by starting with the empty graph on $n$ vertices and considering edges $e_1,\ldots,e_m$ one-at-a-time; each time, we add $e_i$ unless it joins two components that each contain a special vertex.

\medskip

So the number of edges in $\Gk(n,m)$ is not $m$, rather it is $m$ minus the number of edges that were skipped. However, note that $\Gk(n,m={n\choose2})$ is identical to $\Pk(n,M)$ (recall that $M=M(n,k)$ is defined to be the number of edges in $\Pk$ once no more edges can be added; that is, after all ${n\choose2}$ edges are considered). So it suffices to prove Theorem~\ref{thm:main} for $\Gk(n,m={n\choose2})$.

\begin{remark}\label{rcouple}
It will be useful to note that the sequence of edges $e_1,\ldots,e_m$ is independent of the set of special vertices.  By symmetry, we can assume that the special vertices are chosen uniformly from amongst the $n$ vertices, and so we can first choose $e_1,\ldots,e_m$ and then choose the $k$ vertices uniformly.
\end{remark}

Note that we can couple $\Gk(n,m)$ to the random graph process $\G(n,m)$ by using the same sequence $e_1,\ldots,e_m$. This coupling will be valuable as it will allow us to apply some deep and technical results regarding the giant component of $\G(n,m)$. However, there are also some implications that are easy and straightforward, but useful at the same time. 

For example, one can couple $\Gk(n,m)$ and $\mathbb{\hat{\G}}_{k+1}(n,m)$ by making sure that each special vertex in $\Gk(n,m)$ is also special in  $\mathbb{\hat{\G}}_{k+1}(n,m)$. Then, it is clear that a collision in $\Gk(n,m)$ is also a collision in $\mathbb{\hat{\G}}_{k+1}(n,m)$. Moreover, at every step the set of components of $\Gk(n,m)$ differs from the set of components of $\mathbb{\hat{\G}}_{k+1}(n,m)$ only in that possibly two of the special components in the latter process are joined into one component in the former. This yields the following monotonicity result:

\begin{observation}\label{obiggerk} For any $1\leq k_1\leq k_2\leq n$, and for any $m$:
\begin{enumerate}
\item[(a)] the largest special component in $\mathbb{\hat{\G}}_{k_1}(n,m)$ is at least as big as the largest component in $\mathbb{\hat{\G}}_{k_2}(n,m)$; and
\item[(b)] $\hat{M}(n,k_1) > \hat{M}(n,k_2)$.
\end{enumerate}
\end{observation}

Here is another implication. Note that if $\G(n,m)$ is connected, then $\Gk(n,m)$ has $k$ components. Since $\G(n,m)$ is a.a.s.\ connected for $m = n \log n / 2  + n \omega$ (where $\omega=\omega(n)$ is any function tending to infinity as $n \to \infty$), a.a.s.\ $\hat{M}(n,k) \le n \log n / 2 + n \omega$ for any $2 \le k \le n$. Since, $\G(n,m)$ is a.a.s.\ disconnected for $m = n \log n / 2 - n \omega$, this bound is sharp for $k=1$: $\hat{M}(n,1)\sim n \log n / 2$. In fact, it is straightforward to prove that $\hat{M}(n,k) \sim n \log n / 2$ for $k \ll n^{1/3}$. Indeed, a.a.s.\ many (precisely $(1+o(1)) e^{2\omega}$) non-special vertices in $\Gk(n,m)$ are still isolated at time $m = n \log n / 2 - n \omega$; hence, a.a.s.\ there are more than $k$ components in $\Gk(n,m)$. It follows that a.a.s.\ $\hat{M}(n,k)$ is at least $m$ minus the number of collisions up to this point of the process. More importantly, by Theorem~\ref{thm:main}(a), a.a.s.\ the giant component in $\Pk(n,\hat{M})$ has size $n(1-o(1))$, and so a.a.s.\ all collisions that occurred in $\Gk(n,m)$ must involve vertices from a small set of size $o(n)$. Finally, one can show that a.a.s.\ removing any set of size $o(n)$ from $\G(n,m)$ decreases the number of edges by $o(n \log n)$. As this is a trivial bound for the number of collisions in $\Gk(n,m)$, we get that a.a.s.\ $\hat{M}(n,k) \sim  n \log n / 2$, provided that $k \ll n^{1/3}$. We do not provide a formal proof here as it seems that understanding the behaviour of $\hat{M}(n,k)$ for $k \gg n^{1/3}$ require more work and a better understanding of the process. $\hat{M}(n,k)$ continues to decrease as $k$ increases (see Observation~\ref{obiggerk}) and, trivially, $\hat{M}(n,n)=0$. But the behaviour of this is unknown.

\medskip

We will also make use of the $\Gp(n,p)$ model:  we begin with $n$ vertices and then decide to include each of the $n\choose 2$ possible edges independently with probability $p$. A standard and very useful fact is that we can typically translate a.a.s.\ properties between $\G(n,m)$ and $\Gp(n,p)$ when $m\approx p{n\choose 2}$.  For example, in this paper we will use the following, which comes from~(1.6) in~\cite{JLR}. 

\begin{lemma} \label{lem:gnp_to_gnm}
Let $m=m(n)$ be any function such that $m \le n$, let $\gamma=\gamma(n)$ be any function such that $\gamma \gg \sqrt{n}$, and take $p=p(n) = (m-\gamma)/{n \choose 2}$. Then one can couple the two processes such that a.a.s.\ $\Gp(n,p) \subseteq \G(n,m)$.
\end{lemma}
\begin{proof}
Note that $\Gp(n,p)$ can be generated in two steps. First, we expose the total number of edges $M$, the binomial random variable $\bin( {n \choose 2}, p)$ with $\E[M] = m - \gamma$. Then, clearly, $\Gp(n,p) = \G(n,M)$. It follows from Chernoff's bound that
$$
\Pr \left( M \ge m \right) = \Pr \left( M \ge \E[M]+\gamma \right) \le \exp \left( - \frac {\gamma^2}{2m} \right) = o(1),
$$
as $\gamma \gg \sqrt{n}$. We get that a.a.s.\ $\Gp(n,p) = \G(n,M) \subseteq \G(n,m)$, and the proof of the lemma is finished.
\end{proof}

\subsection{Largest component in $\G(n,m)$}\label{sec:known-results}

We will need the following well-known result on the component sizes of $G(n,m)$ when $m$ is close to the critical point $n/2$.  These bounds follow immediately from Theorems~5 and~6 of~\cite{luczak1990}.

\begin{lemma}\label{lem:tluc} 
A.a.s.\ the random graph process is such that:
\begin{enumerate}
\item[(a)] For every integer $m$ where $m=\lfloor\frac{n}{2}(1-\lambda n^{-1/3})\rfloor$ for some $0\ll\lambda\ll n^{1/3}$, the largest component in $\G(n,m)$ has size $\Theta(n^{2/3}\lambda^{-2}\log\lambda)$.

\item[(b)] For every integer $m$ where $m=\lfloor\frac{n}{2}(1+\lambda n^{-1/3})\rfloor$ for some $0\ll\lambda\ll n^{1/3}$
\begin{enumerate}
\item[(i)] the largest component in $\G(n,m)$ has size $(2+o(1))\lambda n^{2/3}$;
\item[(ii)] the second largest component in $\G(n,m)$ has size $\Theta(n^{2/3}\lambda^{-2}\log\lambda)$.
\end{enumerate}
\end{enumerate}
\end{lemma}
The range of $m$ covered in part (a) is referred to as the {\em subcritical range}; the range covered in part (b) is the {\em supercritical range}. We also know that the giant component is formed from smaller ones during the so-called \emph{critical phase} when $m=\lfloor\frac{n}{2}(1+\Theta( n^{-1/3}))\rfloor$. During the critical phase, the largest component  has cardinality of order $n^{2/3}$.

\subsection{Susceptibility}

The susceptibility $\chi(G)$ of a graph $G$ (deterministic or random) is defined as the expected size of the component containing a random vertex. If the list of component sizes is $s_1, s_2, \ldots, s_r$, then
$$
\chi(G) = \sum_{i=1}^r \frac {s_i}{n} s_i = n^{-1} \sum_{i=1}^r s_i^2.
$$
Without loss of generality, we may assume that $s_1 \ge s_2 \ge \ldots \ge s_r$. 

For the supercritical case, one can show that the giant component will dominate all other terms in the sum and so a.a.s.\ $\chi( \G(n,m) ) \sim s_1^2 / n \sim 4 \lambda^2 n^{1/3}$ (see Appendix~A in~\cite{jansonluczak2008}). Similarly, for the critical phase, there are several components of order $n^{2/3}$ but a.a.s.\ $\chi( \G(n,m) ) = \Theta( s_1^2 / n) = \Theta( n^{1/3} )$ (see Appendix~B in~\cite{jansonluczak2008}). The biggest challenge is to analyze the subcritical phase and this is the main focus of~\cite{jansonluczak2008}, where the following is proved (see Theorem~1.1).

\begin{lemma}\label{lem:susceptibility}
If $n/2 - m \gg n^{2/3}$, then a.a.s.
$$
\chi( \G(n,m) ) \sim \frac {n/2}{n/2-m}.
$$
\end{lemma}


\section{Concentration tools}\label{sec:concentration-tools}

Let us start this section with the following result which is a  generalization of a well-known Chernoff bound.

\begin{lemma}\label{lem:chernoff}
Let $\ccc = (c_1, c_2, \ldots, c_r)$ be a sequence of natural numbers with  $c = \max_{i} c_i$. 
Let $S_j = \sum_{i=1}^j c_i Z_i$, where $Z_i, i \in [r]$ are independent Bernoulli($p$) random variables. Let $\mu_j = \E [S_j]=  p\sum_{i=1}^j c_i$, and let $\mu=\mu_r=\E[S_r]$.
Then for $t \ge 0$ we have that
\begin{eqnarray*}
\Prob \left(\max_{1\leq j\leq r} (S_j-\mu_j) \ge t \right) &\le& \exp \left( - \frac {t^2}{2c(\mu+t/3)} \right) \mbox{ and }\\
\Prob \left(\max_{1\leq j\leq r} (\mu_j-S_j) \ge t \right) &\le& \exp \left( - \frac {t^2}{2c \mu} \right). \\
\end{eqnarray*}
In particular, for $\eps \le 3/2$ we have that
$$
\Prob \left( \max_{1\leq j\leq r} |S_j-\mu_j| \ge \eps \mu \right) \le 2 \exp \left( - \frac {\eps^2 \mu}{3c} \right).
$$
\end{lemma}

To prove this lemma, one can easily adjust the proof of the classic Chernoff bound. Alternatively, the same bounds come from~\cite{mcdiarmid1998}. In that paper, the counterpart of Lemma~\ref{lem:chernoff} is stated for $S_r-\mu$ (see Theorem~2.3); however, the author comments that $S_r-\mu$ can be replaced with $\max_{1\leq j\leq r} (S_j-\mu_j)$ (which is a slightly stronger version that we need here) as follows. A standard martingale bound shows that eg.\ for any $h>0$:
\[\Prob \left( \max_{1\leq j\leq r} (S_j-\mu_j) \ge t \right) \leq e^{-ht}\E \left[ e^{h(S_r-\mu_r} \right]. \]
Then plugging this into the appropriate place in the proof of Theorem 2.3 yields the desired bounds.

\subsection{A rich-get-richer process}

Understanding the following process will be crucial in our analysis. Let $x$ and $y$ be any natural numbers (typically $x=x(n)$, $y=y(n)$, and other values defined here are functions of $n$ and tend to infinity as $n\to \infty$). Let $\ccc = (c_1, c_2, \ldots, c_r)$ be a sequence of natural numbers and let $c = \max_i c_i$. Finally, for any $0 \le q \le r$, let $t_q = x+y + \sum_{i=1}^q c_i$. Clearly, $t_q$ is an increasing sequence with $t_0 = x+y$.


We define the \emph{$(\ccc,x,y)$-process} as follows. The process starts with $X({t_0}) = x$ and $Y({t_0})=y$. For any $1 \le q \le r$, with probability $p_q$ where
\begin{equation}\label{epqdef}
p_q := \frac {X(t_{q-1})}{X(t_{q-1})+Y(t_{q-1})},
\end{equation}
the two random variables are updated as follows:
\begin{eqnarray*}
X(t_q) &=& X(t_{q-1}) + c_q  \\
Y(t_q) &=& Y(t_{q-1});
\end{eqnarray*}
otherwise,
\begin{eqnarray*}
X(t_q) &=& X(t_{q-1}) \\
Y(t_q) &=& Y(t_{q-1}) + c_q.
\end{eqnarray*}
Note that for any $0 \le q \le r$ we have $X(t_q) + Y(t_q) = t_q$. 

\medskip

In expectation, the ratio $p_q$ does not change throughout the process. Indeed,
\begin{eqnarray*}
\E[ p_{q+1} ~~|~~ p_q ] &=& p_q \frac {X(t_{q-1})+c_q}{X(t_{q-1})+Y(t_{q-1})+c_q} + (1-p_q) \frac {X(t_{q-1})}{X(t_{q-1})+Y(t_{q-1})+c_q} \\
&=& \frac {p_q c_q + X(t_{q-1})}{X(t_{q-1})+Y(t_{q-1})+c_q}\\
&=& \frac {p_q (c_q + X(t_{q-1})+Y(t_{q-1}))}{X(t_{q-1})+Y(t_{q-1})+c_q} \qquad \qquad \mbox{by~(\ref{epqdef})} \\ 
&=& p_q.
\end{eqnarray*}
The following lemma shows that this ratio is concentrated around $p_1 = x / (x+y)$.

\begin{lemma}\label{lem:key}
Consider the $(\ccc,x,y)$-process for some sequence $\ccc$ and natural numbers $x, y$ such that $\sum_{i=1}^rc_i<\frac{1}{2}(x+y)$. Then, for any $w \ge 1$,
$$
\Prob \left( \frac {X(t_r)}{t_r} > \frac{x}{x+y} \left( 1 + \frac {1}{w}\right) \right) \le \Prob \left( X(t_r) > \frac{x}{x+y}t_r + \frac {x}{w} \right) \le \exp \left( - \frac {x}{12 c w^2} \right).
$$
\end{lemma}
\begin{proof} 
Note that if $X(t_{q-1}) \leq \frac{x}{x+y}t_{q-1} + \frac {x}{w}$  then 
\begin{eqnarray}
p_q &=& \frac {X(t_{q-1})}{X(t_{q-1})+Y(t_{q-1})} \nonumber \\
&\le& \frac {x (t_{q-1}/t_0+1/w)}{t_{q-1}} = \frac{x}{t_0} \left( 1 + \frac {t_0/t_{q-1}}{w} \right) \le \frac{x}{t_0} \left( 1 + \frac {1}{w} \right) =: p.\label{epbound}
\end{eqnarray}
We define $Z_1,\ldots,Z_r$ to be independent Bernoulli$(p)$ variables. Let $X'(t_0)=X(t_0)=x$. For each $1\leq q\leq r$ we define $X'(t_q)=X(t_q)$ if $X(t_{q'}) \leq \frac{x}{x+y}t_{q'} + \frac {x}{w}$ for every $q'<q$ and otherwise $X'(t_q)=X'(t_{q-1})+c_q Z_q$. Thus,  defining $S_q=\sum_{i=1}^q c_i Z_i$, (\ref{epbound}) implies that we can couple the process with $S_q$ so that $X'(t_q)\leq x+S_q$  for all $0\leq q\leq r$.

So we can apply Lemma~\ref{lem:chernoff} to bound the probability that $X'$ ever deviates much from its mean. Noting that $\mu_{r} =  p(t_{r}-t_0) \le pt_0/2 = (x/2) (1+1/w) \le x$ since $w\geq 1$, and setting $t = x/(2w) \le x/2$ we get that
\begin{eqnarray}
\nonumber\Prob \left( \max_{1\leq q\leq r} (S_q-\mu_q) >  \frac {x}{2w} \right) &\le& \exp \left( - \frac {t^2}{2c(\mu_{r}+t/3)} \right) \\
&\le& \exp \left( - \frac {(x/(2w))^2}{2c(x+(x/2)/3)} \right) =  \exp \left( - \frac {x}{12cw^2} \right).
\label{esq}
\end{eqnarray}
This implies the lemma since if $Q$ is the smallest $q\leq r$ for which $X(t_{q}) > \frac{x}{x+y}t_{q} + \frac {x}{w}$ then $X'_q=X_q$ for all $q\leq Q$ by definition and so 
\[S_Q\geq X'(t_Q)-x>  \frac {x}{T}(t_Q-T) + \frac {x}{w}=p(t_Q-T) + \frac {x}{w} \cdot \frac {2 T - t_Q}{T}
> p(t_Q-T) + \frac {x}{2w},\]
since $t_Q\leq t_r\leq\frac{3}{2}T$.
\end{proof}

We finish this subsection with the following result.

\begin{lemma}\label{lem:cxy_process}
Suppose that $c,x \ll y$. Then, at the end of the $(\ccc,x,y)$-process described above,  with probability at least $1-\exp (- x/(20c) )$,
$$
X{(t_r)} = O(xt_r/y).
$$
\end{lemma}

Before we move to the proof of the lemma, note that we can assume that $c\ll x$ as otherwise we can replace $x$ by some $x'>x$ with $c\ll x'\ll y$.  Clearly, running $(\ccc,x',y)$ rather than $(\ccc,x,y)$ only decreases the probability that $X\leq Z$ for any given $Z$; in particular, it decreases the probability that $X(t_r)=O(xt_r/y)$. So if the lemma holds for $x'$ then it holds for $x$.

\begin{proof}
Lemma~\ref{lem:key} requires that $\sum_{i=1}^r c_i\leq (x+y)/2$.  In order to apply this lemma, we split the process into phases. To simplify the notation, set $r_0=0$. For the first phase we take the longest sub-sequence $\ccc_1 = (c_1, c_2, \ldots, c_{r_1})$ such that $\sum_{i=1}^{r_1} c_i \le (x+y)/2$. Next, we pick the longest subsequence $\ccc_2 = (c_{r_1+1}, c_{r_1+2}, \ldots, c_{r_2})$ such that $ \sum_{i={r_1+1}}^{r_2} c_i \le \frac12(x+y+\sum_{i=1}^{r_1} c_i)$, and so on, for each $j$ picking the longest subsequence $\ccc_j = (c_{r_{j-1}+1}, c_{r_{j-1}+2}, \ldots, c_{r_{j}})$ such that 
\begin{equation}\label{eCjsum}
\sum_{i={r_{j-1}+1}}^{r_j} c_i \le \frac12\left(x+y+\sum_{i=1}^{r_{j-1}} c_i\right).
\end{equation} 
The last phase, phase $\ell$, deals with the sequence $\ccc_\ell = (c_{r_{\ell-1}+1}, c_{r_{\ell-1}+2}, \ldots, c_{r_\ell}=c_r)$. Now the $(\ccc,x,y)$-process can be treated as a series of $\ell$ processes, each on the sequence $\ccc_j$ and with initial values taken from the end of the previous sequence; i.e.\ phase $j$ is the $(\ccc_j, X(t_{r_{j-1}}), Y(t_{r_{j-1}}))$-process, where $X(t_0):= x, Y(t_0):= y$.  

Recall that for every $q$, $X(t_q)+Y(t_q)=x+y+\sum_{i=1}^q c_i$. So~(\ref{eCjsum}) implies that we can apply Lemma~\ref{lem:key} to each phase.  For any $j \in \nat$, let $w_j = 1.1^{j}$, and for any $j \in \nat \cup \{0\}$, let
$$
{E_j = \prod_{i=1}^{j} \left( 1 + \frac {1}{w_i} \right) = \Theta \left( 1 \right),} 
$$
since 
$$
{1 \le \prod_{i=1}^{j} \left( 1 + \frac {1}{w_i} \right) \le \exp \left( \sum_{i=1}^{\infty} 1/w_i \right) = e^{10}. }
$$
(In particular, $E_0 = 1$.) We will prove that with probability at least $1-\exp (- x/(20c) )$, at the end of every phase $j$ we have:
\begin{equation}\label{exej}
{ X(t_{r_j})\leq \frac {x}{x+y} E_{j} t_{r_j} }.
\end{equation} 
If~(\ref{exej}) holds at the end of phase $j-1$ (or if $j=1$; note that~(\ref{exej}) trivially holds for $j=0$), then the probability that~(\ref{exej}) holds at the end of phase $j$ is at least the probability that it holds if we adjust the initial values to
\[x_j:= \frac {x}{x+y} E_{j-1} t_{r_{j-1}}; \qquad y_j:= t_{r_{j-1}} -x_j.\]
So we get a lower bound on the probability of~(\ref{exej}) by applying Lemma~\ref{lem:key} to  the $(\ccc_j, x_j, y_j)$-process. 

We require two bounds. For the first one, recall that by definition $t_q = x+y + \sum_{i=1}^q c_i$. Hence, $t_{r_0}=t_0=x+y$ and, since we choose the longest $\ccc_j$ satisfying~(\ref{eCjsum}) and each $c_i\leq c\ll x+y\leq  t_{r_{j-1}}$, for each $1\leq j<\ell$ we have 
\[t_{r_{j}} = t_{r_{j-1}} + \sum_{i={r_{j-1}+1}}^{r_j} c_i> 1.5 t_{r_{j-1}} - c = (1.5-o(1)) t_{r_{j-1}} > 1.4 t_{r_{j-1}}> 1.4^j (x+y).\]
Therefore,
\[ 
x_j > x(1.4)^{j-1}E_{j-1} \ge x(1.4)^{j-1} \qquad \qquad \qquad \text{ for any } 1 \le j \le \ell
\]
and, since $c\ll x$, 
\[\frac{x}{17c}\left((1.1)^j-1\right)>j.
\]
Note that, since $x_j+y_j=t_{r_{j-1}}$,
\[\frac{x_j}{x_j+y_j}\left(1+\frac{1}{w_j}\right) t_{r_j}
=\frac {x}{x+y} E_{j-1} t_{r_{j-1}}\times\frac{1}{t_{r_{j-1}}}\times\left(1+\frac{1}{w_j}\right) t_{r_j}
=\frac {x}{x+y} E_{j} t_{r_j}.
\]
So Lemma~\ref{lem:key} (applied with $x=x_j$, $y=y_j$, and $t_r=t_{r_j}$; $c$ remains the same for all applications of the lemma) yields that the probability that~(\ref{exej}) fails to hold at the end of phase $j$ is at most 
$$
\exp \left( - \frac {x_{j}}{12 c w_j^2} \right) \le \exp \left( - \frac {x 1.4^{j-1}}{12 c 1.1^{2j}} \right) \le \exp \left( - \frac {x 1.1^{j}}{17 c} \right) \le 2^{-j} \exp (- x/(17c)).
$$
Hence, since $\sum_{j \ge 0} 2^{-j} \exp (- x/(17c)) \le 2 \exp (- x/(17c) ) \le \exp (- x/(20c) ) \to 0$, with the desired probability~(\ref{exej}) holds at the end of every phase. Since the last phase ends at $r_{\ell}=r$ and $E_{\ell}=\Theta(1)$, this implies
\[
 X(t_r) = X(t_{r_{\ell}})\leq \frac {x}{x+y} E_{\ell} t_{r_{\ell}} = \Theta\left(\frac{x t_r}{x+y}\right) =O\left(\frac{xt_r}{y}\right),
\]
as $x\ll y$.
\end{proof}

\section{The giant has enough time to be born: $k \ll n^{1/3}$}\label{sec:k-small}

Suppose that $k \ll n^{1/3}$.  As mentioned earlier, we will prove that for this range of the parameter $k$, the giant component is formed before collisions start affecting the process. In particular, when the first special vertex joins the giant its size is much larger than the total size of all other special components---see Lemma~\ref{lem:prop_m2}. As a result, the giant will continue growing and at the end of the process it will have size $n-o(n)$---see Theorem~\ref{thm:main}(a).

\subsection{Early phase}

Let $\omega=\omega(n)$ be any function that grows with $n$ sufficiently slowly to satisfy various bounds that follow. In particular, it will grow more slowly than $n^{1/3}/k$; {we may then assume that $\omega^2 \le n^{1/3}/k$. 

It will be also convenient to assume that $k = k(n)$ tends to infinity faster than $\omega$ so let us assume for now that $k \ge \omega^2$; we will discuss how to translate the results to other values of $k$ (including the case when $k$ is a constant) at the end of this section.} Define:
\begin{eqnarray*}
\lambda_1 &=& \lambda_1(n) = n^{1/3}/(k \omega) \ll n^{1/3}/k \\
 m_1&=& (n/2) (1+ \lambda_1 n^{-1/3}).
 \end{eqnarray*}
So $m_1\sim n/2$ and  $\G(n,m_1)$ is in the supercritical phase (note that $\omega$ tends to infinity slowly enough so that $\lambda_1$ tends to infinity).

\medskip

Recall that for a given graph $G$ with $r$ connected components, $L_i(G)$ is the size of an $i$-th largest component ($i=1,2, \ldots, r$). Similarly, let $\hat{L}_i(G)$ be the size of an $i$-th largest special component ($i=1,2,\ldots, k$). 

\medskip

Let us start with the following observation.

\begin{lemma}\label{lem:prop_m1}
Suppose that $\omega^2 \le k \le n^{1/3}/\omega^2$ for some $\omega=\omega(n) \to \infty$ as $n \to \infty$. Let $\lambda_1$ and $m_1$ be defined as above. Then a.a.s.\ the following properties hold.
\begin{itemize}
\item [(a)] $L_1(\Gk(n,m_1)) \sim 2n / (k \omega) = o(n)$;
\item [(b)] $L_2(\Gk(n,m_1)) = \Theta(k^2 \omega^2 \log (n^{1/3}/k))$;
\item [(c)] $\hat{L}_1(\Gk(n,m_1)) = O(k^2 \omega^2 \log (n^{1/3}/k))$.
\end{itemize}
\end{lemma}
\begin{proof}
The proof follows easily {from Lemma~\ref{lem:tluc}} applied to $\G(n,m_1)$. We get that a.a.s.\ the complex component of $\G(n,m_1)$ has size asymptotic to $2 \lambda_1 n^{2/3} = 2n / (k \omega) = o(n)$. Moreover, a.a.s.\ the size of the second largest component is of order $n^{2/3} \lambda_1^{-2} \log \lambda_1 = \Theta(k^2 \omega^2 \log (n^{1/3}/k))$. Since we aim for a statement that holds a.a.s.\ we may assume that $\G(n,m_1)$ has these properties. Now, we select $k$ special vertices at random to translate these observations to $\Gk(n,m_1)$, as described in Remark~\ref{rcouple} in Section~\ref{sec:more_definitions}. The expected number of special vertices that belong to the complex component is {asymptotic to} $k\cdot \frac{2n}{k \omega} \cdot\frac{1}{n}=o(1)$, so a.a.s.\ no special vertex belongs there. This implies (a) and (c). The same argument shows that a.a.s.\ no special vertex belongs to the second largest component which implies (b). The proof of the lemma is finished.
\end{proof}

So we can assume that $\Gk(n,m_1)$ satisfies the properties of Lemma~\ref{lem:prop_m1}. In particular, the largest component of $\Gk(n,m_1)$ contains no special vertices and so it is identical to the largest component of $\G(n,m_1)$ under the coupling described in Section~\ref{sec:concentration-tools}.  Since $\Gk(n,{n\choose 2})$ has exactly $k$ components, each with one special vertex, there will be a step of the $k$-process when the largest component is joined to a component containing a special vertex; we define this step as:
\begin{align*}
m_2 & \mbox{ is the first step following $m_1$} \\
& \mbox{ in which the largest component of $\G(n,m_2)$ contains a special vertex.}
\end{align*}


It is worth noting that a.a.s.\ $m_2\sim n/2$; we only sketch the straightforward proof.
One can show easily, using Remark~\ref{rcouple}, that a.a.s.\ the largest component of $\G(n,m=(n/2) (1+ \lambda_2 n^{-1/3}))$ contains at least one special vertex if
$$
\lambda_2 = \lambda_2(n) = n^{1/3} \omega/k \gg n^{1/3}/k.
$$
Therefore, {a.a.s.}\ $m_2 \le (n/2) (1+ \lambda_2 n^{-1/3})$. Hence, {a.a.s.}\ $m_2 \sim n/2$ (since it is assumed that $k \ge \omega^2$).

\begin{lemma}\label{lem:prop_m2}
Suppose that {$\omega^2 \le k \le n^{1/3}/\omega^2$ for some $\omega=\omega(n) \to \infty$ as $n \to \infty$}. Let $m_2$ be defined as above. Then, a.a.s.\ the following properties hold
\begin{itemize}
\item [(a)] $\hat{L}_1(\Gk(n,m_2)) = L_1(\Gk(n,m_2)) \ge n / (k \omega)$;
\item [(b)] $\sum_{i=2}^k \hat{L}_i(\Gk(n,m_2)) \le k^2\omega^3\log(n^{1/3}/k)$.
\end{itemize}
Moreover, for every $m \ge m_2$:
\begin{itemize}
\item [(c)] the size of any non-special component in $\Gk(n,m)$ is at most $ k^2 \omega^3 \log (n^{1/3}/k)$.
\end{itemize}
\end{lemma}
\begin{proof}
It follows from Lemma~\ref{lem:prop_m1}(a) that the giant component of $\Gk(n,m_1)$ has size at least $ n / (k \omega)$. It keeps growing from that point on and so the same lower bound holds at time $m_2$. Property~(a) trivially holds. Moreover, using Lemma~\ref{lem:tluc}, not only at time $m_1$ (as indicated by Lemma~\ref{lem:prop_m1}(b)) but also if one continues the random graph process from time $m_1$ on, a.a.s.\ the size of the second largest component of $\G(n,m)$ is always at most $\Theta(k^2 \omega^2 \log (n^{1/3}/k))<k^2 \omega^3 \log (n^{1/3}/k)$. Since any non-special component in $\Gk(n,m)$ is a component in $\G(n,m)$, this proves property~(c). 

Now, in order to show that property (b) holds, we must study the subgraph induced by the vertices not in the largest component. We define:
\[\G^L(n,m) \mbox{ is the graph obtained by deleting the largest component from } \G(n,m).\] 
This is of particular interest when $m$ is in the supercritical range. We define:
\begin{eqnarray*}
n'&=&\mbox{ the number of vertices in } \G^L(n,m_2); \\
m'&=&\mbox{ the number of edges in } \G^L(n,m_2); \\
k' &=& \mbox{ the number of special vertices in } \G^L(n,m_2).
\end{eqnarray*}
By Lemma~\ref{lem:tluc} we have a.a.s.\ 
\begin{equation}\label{en'm'}
n'=n-(4+o(1))(m_2-n/2);\  m'=n'/2-(1+o(1))(m_2-n/2);
\end{equation}
and the component of $\G(n,m_2-1)$ that is added to $\Theta^*$ in step $m_2$ contains exactly $k-k'\geq 1$ special vertices.

Conveniently, the distribution of $\G^L(n,m_2)$ and of its $k'$ special vertices is nearly uniform, despite the conditioning implied by the definition of $m_2$.  Formally, we need the following claim.

\medskip

{\bf Claim:} {\em Expose the values of $m_2, k'$, and the largest component of $\G(n,m_2)$; denote that largest component by $\Theta^*$. (Note that this determines the vertex set of $\G^L(n,m_2)$ and $m'$.) Conditional on that exposure:
\begin{enumerate}
\item[(i)] Every graph on the $n'$ vertices of $\G^L(n,m_2)$ that has $m'$ edges and  no component at least as large as $\Theta^*$ is equally likely to be $\G^L(n,m_2)$.
\item[(ii)] Every set of $k'$ vertices in $\G^L(n,m_2)$ is equally likely to be the special vertices.
\end{enumerate}}

{\bf Proof of the claim:} Consider (i) any set $S$ of $k$ special vertices where exactly $k'$ are outside of $\Theta^*$, and (ii) any random graph process $e_1,\ldots,e_{m_2}$ in which $m_2$ is the first step following $m_1$ where the largest component contains a member of $S$ and $\Theta^*$ is that largest component. Let $L$ be the graph formed by removing $\Theta^*$.

Let $L'$ be any graph on the same vertex set as $L$ with $m'$ edges and with no component larger than $\Theta^*$. Replace the edges of $L$ in the process with the edges of $L'$, in any order; let $G'_i$ be the graph formed by the first $i$ edges of the resulting sequence.  Replace the $k'$ special vertices in $L$ by any set of $k'$ vertices in $V(L)$, and do not change the $k-k'$ special vertices in $\Theta^*$; denote the resulting set of $k$ special vertices as $S'$.  It is straightforward to check that (1) the largest component of $G'_{m_2-1}$ contains no vertex of $S'$, and (2) the largest component of $G'_{m_2}$ is $\Theta^*$ and hence contains a vertex of $S'$.   So $m_2, k'$ and the largest component at step $m_2$ are the same in both processes. Furthermore, each sequence of edges is equally likely to be selected. This implies the claim.

\medskip

By part (ii) of our claim, and reasoning like that in Remark~\ref{rcouple}, we can first expose the graph $\G^L(n,m_2)$ and then choose the $k'$ special vertices. Note that the expected total size of the components containing those vertices is $k'\chi(\G^L(n,m_2))$.
By part (i), we can treat $\G^L(n,m_2)$ as $\G(n',m')$ which, by Lemma~\ref{lem:susceptibility} and~(\ref{en'm'}) a.a.s.\ has susceptability:
\[\chi(\G^L(n,m_2)) \sim \frac {n'/2}{n'/2-m'} \sim k \omega.\]
So the expected total size of the components of $\G^L(n,m_2)$ containing special vertices is $k'\times (1+o(1))k\omega\leq(1+o(1))k^2\omega$. It follows from Markov's inequality that the total size is a.a.s.\ at most $k^2 \omega^2$. 

This bounds the total size of all special components, other than the largest, in $\G(n,m_2)$.  But we actually need to bound the total size in $\Gk(n,m_2)$. If $k'=k-1$, i.e.\ if only one special vertex joins the largest component in step $m_2$, then these two totals are the same.  Otherwise, let $\Phi$ be the component of $\G(n,m_2-1)$ that contains $k-k'$ special vertices and is merged with the largest component in step $m_2$.  In $\Gk(n,m_2-1)$, $\Phi$ is partitioned into exactly $k-k'$ components.  One of them is joined to the largest component in step $m_2$; the others have total size at most $|\Phi|$ which, by Lemma~\ref{lem:tluc}(b) and since $m_2>m_1$, is at most
$\Theta(n^{2/3}\lambda_1^{-2}\log\lambda_1)=\Theta(\omega^2k^2\log(n^{1/3}/k\omega)).$
$\hat{L}_2(\Gk(n,m_2)), \ldots, \hat{L}_k(\Gk(n,m_2))$ consists of those $k-k'-1$ components, along with the $k'$ special components contained in the components of  $\G^L(n,m_2)$ that contain special vertices.  The total size of the latter set was bounded above, and so
\[\sum_{i=2}^k \hat{L}_i(\Gk(n,m_2)) \le k^2 \omega^2 + \Theta(\omega^2k^2\log(n^{1/3}/k\omega)) <\omega^3k^2\log(n^{1/3}/k),\]
thus proving part (b).
\end{proof}

\subsection{Modelling $\Gk(n,m)$ with a $(\ccc,x,y)$-process}\label{sec:connection_to_cxy}

{We continue assuming that $\omega^2 \le k \le n^{1/3}/\omega^2$ for some $\omega=\omega(n) \to \infty$ as $n \to \infty$.}
Beginning at time $m=m_2$, we do not consider $\G(n,m)$  and instead focus directly on $\Gk(n,m)$.  Recall that a component is called {\em special} if it contains a special vertex, and so we always have exactly $k$ special components.

Let $j_1,\ldots,j_{\ell} > m_2$ denote the steps in the $\G_k$-process during which we choose an edge joining a special component to a non-special component.  Of course, we accept that edge.  Let $\Theta_i$ be the non-special component chosen at time $j_i$, and set $c_i=|\Theta_i|$.

\begin{observation} 
After $m=m_2$, the sizes of the special components only change during steps $j_1,\ldots,j_{\ell}$.
\end{observation}

Now expose the components $\Theta_1,\ldots,\Theta_{\ell}$ but not the edges selected at times $j_1,\ldots,j_{\ell}$.

\begin{observation} 
Conditional on any choice for $\Theta_1,\ldots,\Theta_{\ell}$, at each step $j_i$, the probability that $\Theta_i$ is joined to a particular special component is proportional to the size of that special component.
\end{observation}

So we can model the growth of the largest component with a $(\ccc,x,y)$ process.  Let
\[y=L_1(\Gk(n,m_2))=\hat{L}_1(\Gk(n,m_2));\qquad x=\sum_{i=2}^k \hat{L}_i(\Gk(n,m_2));\]
i.e.\ $y$ is the size of the largest special component and $x$ is the total size of all other special components at step $m_2$. Then setting $\ccc=(c_1,\ldots,c_{\ell})$ we see that our two observations yield:

\begin{observation}
The size of the largest special component at steps $j_1,\ldots,j_{\ell}$ follows {random variable $Y$ in} the $(\ccc,x,y)$ process.
\end{observation}

Note that in this process, we have $t_{\ell}=n$.
By Lemma~\ref{lem:prop_m2}, we have:
\[y\geq n/(k\omega); \qquad x\leq k^2\omega^3\log(n^{1/3}/ k);\qquad \mbox{ and } c_i\leq c:= k^2 \omega^3 \log (n^{1/3}/k)\mbox{ for every } i.\]
From that, it is easily verified that $x \le c \ll y$. Indeed, since $k\le n^{1/3}/\omega^2$, we have $c/y\leq k^3\omega^4\log (n^{1/3}/k)/n \le 2 \log \omega / \omega^{2} = o(1)$. Therefore, we can apply Lemma~\ref{lem:cxy_process} to show that at the end of the process, a.a.s.\ the total size of all but the largest special component is 
\[X(n)= O(xn/y)=O( k^3 \omega^4 \log (n^{1/3}/k) ) = o(n).\] 

\medskip

This proves {Theorem~\ref{thm:main}(a), provided that $\omega^2 \le k \le n^{1/3}/\omega^2$. To extend the result to smaller values of $k$, we apply Observation~\ref{obiggerk}. Fix any $k < k':= \omega^2$. Our bound above yields that a.a.s.\ the largest special component at the end of the $k'$-process has size $1 - O( k'^3 \omega^4 \log (n^{1/3}/k') ) = 1 - O( k \omega^{10} \log (n^{1/3}/k) )$. Observation~\ref{obiggerk} implies that the same bound holds for all $2 \le k < k'$, thus proving Theorem~\ref{thm:main}(a). (Note that in the statement of the theorem, we replaced $\omega^{10}$ by $\omega$ which is allowed as in the statement $\omega$ is any function tending to infinity, regardless how slowly it does so.)}

\section{No component has a chance to become giant:\\ $k \gg n^{1/3} (\log n)^{4/3} (\log \log n)^{1/3}$}\label{sec:k-large}

Suppose now that $k \gg n^{1/3}$. As mentioned earlier, for this range of parameter $k$, collisions will start affecting the process much earlier, namely, when each component has size smaller than the total size of all special components---see Lemma~\ref{lem:prop_m3}. Intuitively, this results in no one component dominating the process, and so no component will be able to grow to linear size---see {Theorem~\ref{thm:main} (b)}. In order to prove that this happens we require a stronger bound on $k$, namely:
\begin{equation}\label{egg} 
k \gg n^{1/3} (\log n)^{4/3} (\log \log n)^{1/3}.
\end{equation}
For technical reasons, we also require the following upper bound
\begin{equation}\label{ell} k\ll n/\log n.
\end{equation}
Thus we have the range of $k$ for Theorem~\ref{thm:main}(b).

\subsection{Early phase}

For this section, we define:
\begin{eqnarray*}
\lambda_3 &=& \left( \frac {k}{n^{1/3}} \log( k / n^{1/3}) \right)^{1/2} \\
m_3&= &(n/2) (1-\lambda_3 n^{-1/3}).
\end{eqnarray*}
Note that, since $n^{1/3} \ll k \ll n/\log n$, we have $\lambda_3\rightarrow\infty$ and $m_3=(n/2)(1-o(1))$.  So we are in the subcritical phase.

\begin{lemma}\label{lem:prop_m3}
Suppose that $k \gg n^{1/3}$ and $k \ll n / \log n$. Let $\lambda_3$ and $m_3$ be defined as above. Then a.a.s.\ the following properties hold
\begin{itemize}
\item [(a)] $\hat{L}_1(\Gk(n,m_3)) \le L_1(\Gk(n,m_3)) = \Theta(n/k) \ll n^{2/3}$;
\item [(b)] $\sum_{i=1}^k \hat{L}_i(\Gk(n,m_3)) \ge  \frac 14 \left( \frac {n k}{\log(k/n^{1/3})} \right)^{1/2} \gg n^{2/3}$.
\end{itemize}
Moreover, for any $m \ge m_3$:
\begin{itemize}
\item [(c)] the size of any non-special component in $\Gk(n,m_3)$ is at most $n \log n / k$.
\end{itemize}
\end{lemma}

Since the proof of this lemma is long, we split it into a few parts. 

\begin{proof}[Proof of Lemma~\ref{lem:prop_m3}(a)]
From Lemma~\ref{lem:tluc} we get that a.a.s.\ the size of the largest component in $\G(n,m_3)$ is equal to
$$
x = \Theta(n^{2/3} \lambda_3^{-2} \log \lambda_3) = \Theta(n/k).
$$
The size of the largest component in $\Gk(n,m_3)$ is at most the size of the largest component in $\G(n,m_3)$, so property~(a) holds a.a.s.
\end{proof}

\begin{proof}[Proof of Lemma~\ref{lem:prop_m3}(b)]
Note that Lemma~\ref{lem:susceptibility} implies that a.a.s.\ the susceptibility of $\G(n,m_3)$ is 
$$
\chi = \frac {n/2}{n/2 - (n/2 - \lambda_3 n^{2/3}/2)} = \frac {n^{1/3}}{\lambda_3} = \left( \frac {n}{k \log(k/n^{1/3})} \right)^{1/2}.
$$
Hence, the expected total size of all special components is
$$
k \chi = \left( \frac {nk}{\log(k/n^{1/3})} \right)^{1/2}.
$$
If we were only concerned with $k \gg \sqrt{n \log n}$, the concentration of this total size would follow easily from the Hoeffding-Azuma inequality for martingales. But we need a more sophisticated argument for $k$ near $n^{1/3}$. 


It will be simpler to work with the binomial random graph $\Gp(n,p)$ and then translate the results back to $\G(n,m_3)$. Let $\gamma = \gamma(n)= \sqrt{n} \log n \gg \sqrt{n}$, and let
$$
p = \frac {m_3 - \gamma}{{n \choose 2}} = \frac {1-\lambda_3 n^{-1/3}-2\gamma/n}{n-1} = \frac {1 - (1+o(1)) \left( k \log (k/n^{1/3}) / n \right)^{1/2}}{n-1}.
$$

We start with $n$ isolated vertices, $k$ of them are special and form set $K$. We will find a lower bound for the sum of the sizes of all components in $\Gp(n,p)$ containing special vertices. Lemma~\ref{lem:gnp_to_gnm} will then imply that the same bound holds in $\G(n,m_3)$ and so also in $\Gk(n,m_3)$ (as this random variable is exactly the same in both models).

Consider the \emph{breadth-first-search} process starting from $K$. Put all vertices of $K$ into a \emph{queue} $Q$ (first-in first-out list); in any order. Call all vertices of $K$ \emph{saturated}, and then do the following as long as $Q$ is not empty: remove $w$ from $Q$, expose all edges from $w$ to non-saturated vertices, put all new neighbours of $w$ into $Q$ and call them saturated. Note that all saturated vertices lie in special components.

Let $t$ be the random step at which this process halts; i.e.\ reaches $Q=\emptyset$.  For all $i\leq t$ we let $S_i$ denote the number of saturated vertices at step $i$ of the process. 
Set 
\[
s= \frac 14 \left( \frac {n k}{\log(k/n^{1/3})} \right)^{1/2}=\frac 14 k\chi.
\]
It suffices to prove that a.a.s.\ we will reach a step $i$ for which $S_i\geq s$.
Note that at any step $i\leq t$, $|S_i|\geq i$. So it suffices to prove that a.a.s.\ we do not have: 
\begin{equation}\label{em3b}
t< s \qquad \mbox{ and } \qquad S_i\leq s \qquad \forall i\leq t.
\end{equation}
Let $Z_i$ denote the random variable counting the number of vertices added into $Q$ at the $i$th step of the process. Since we remove one vertex from $Q$ at each step, the size of $Q$ at the end of step $i$ is
$$
k-i+ \sum_{j=1}^i Z_j.
$$
Note that $Z_i$ has binomial distribution $\bin(n-S_i,p)$. Indeed, if we let $X_1,\ldots$ be a sequence of independent Bernoulli($p$) variables, then we can couple so that for all $i\leq t$ we have 
\[\sum_{j=1}^i Z_j=\sum_{j'=1}^{\sum_{\ell=1}^in-S_{\ell}}X_{j'}.\]  
So the probability that~(\ref{em3b}) holds is at most the probability that
\begin{equation}\label{e9}
\exists i<s \mbox{ such that } \sum_{j'=1}^{i(n-s) } X_{j'} \leq i-k.
\end{equation}
However, since 
\begin{eqnarray*}
\E\left[\sum_{j'=1}^{i(n-s)} X_{j'} \right] &=& i(n-s)p\\
&=& i \left( 1 - \frac 14 \left( \frac {k}{n \log(k/n^{1/3})} \right)^{1/2} \right) \left( 1 - (1+o(1)) \left( \frac {k\log(k/n^{1/3})}{n} \right)^{1/2} \right) \\
&=& i \left( 1 - (1+o(1)) \left( \frac {k\log(k/n^{1/3})}{n} \right)^{1/2} \right),
\end{eqnarray*}
(\ref{e9}) would imply that 
\begin{align*}
\E\left[\sum_{j'=1}^{i(n-s)} X_{j'} \right] &-\sum_{j'=1}^{i(n-s)} X_{j'}
>k - { (1+o(1)) i \left( \frac {k\log(k/n^{1/3})}{n} \right)^{1/2} } \\
&>k- (1+o(1)) s \left( \frac {k\log(k/n^{1/3})}{n} \right)^{1/2} = \left( \frac 34 - o(1) \right) k  ~~>~~ \frac12 k.
\end{align*}
We note that {$\E\left[\sum_{j'=1}^{s(n-s)} X_{j'} \right]< s <\sqrt{nk}$ } and apply Lemma~\ref{lem:chernoff} with $c=1$ to obtain that the probability of~(\ref{e9}) is at most
\[\exp\left(- \frac {(k/2)^2}{2\sqrt{nk}}\right) =\exp\left(- \frac {(k^3/n)^{1/2}}{8}\right) =o(1),\]
since $k \gg n^{1/3}$. This proves part (b).
\end{proof}

\begin{proof}[Proof of Lemma~\ref{lem:prop_m3}(c)]
Set $c = n \log n / k$. Note first that any non-special component in $\Gk(n,m)$ is a component in $\G(n,m)$. We will run the $\G(n,m)$ process and say that round $m$ is \emph{dangerous} if an edge added during this round connects two components of corresponding sizes $c_1$ and $c_2$, such that $c_1 \le c$, $c_2 \le c$, but $c_1 + c_2 > c$. We say that a dangerous round is \emph{deadly} if the component formed contains no special vertex. We need to show that a.a.s.\ there are no deadly rounds. Clearly, the number of dangerous rounds is at most $n/c$. To bound the probability that the $i$th dangerous round is deadly, we run the $\G(n,m)$ process until the $i$th dangerous round; note that the process up to this point is independent of the choice of special vertices, so we can choose them after the $i$th dangerous round. The probability that none of the $k$ vertices are in the component of size at least $c$ formed in this round is at most
$$
 \frac {{n-c \choose k}} {{n \choose k}} =  \frac {(n-c)_k}{(n)_k} \le \left( 1 - \frac {c}{n} \right)^k \le  \exp \left( - \frac {ck}{n} \right) = \frac {1}{n}.
$$
So the expected number of deadly rounds is at most $\frac{n}{c} \cdot \frac {1}{n} = o(1)$ and so a.a.s.\ no dangerous round is deadly which finishes the proof of this property.
\end{proof}

\subsection{Late phase}
 
We still assume that $k \ll n / \log n$ so that Lemma~\ref{lem:prop_m3} can be applied. We continue the $\Gk(n,m)$ process from time $m_3$ on. We model it with a $(\ccc,x,y)$-process as in Subsection~\ref{sec:connection_to_cxy}. Again, we define $c_1,\ldots,c_{\ell}$ to be the sizes of the non-special components that are joined to special components after step $m_3$.  By Lemma~\ref{lem:prop_m3}(c), every $c_i\leq c:=n \log n / k$.

Let $v$ be any of the $k$ special vertices.  We will let $X$ count the size of the component containing $v$, and we let $Y$ count the total size of the other $k-1$ special components.  By Lemma~\ref{lem:prop_m3}(b,c), initially (i.e., at step $m_3$) we have
\[
X\leq \Theta(n/k) \le x:= 30 n (\log n)^2 / k; \qquad Y\geq y:=\Theta( (nk/\log(k/n^{1/3}))^{1/2} ).
\] 
{(Note that we used a loose upper bound for $X$ to make some room for an argument below that gives the desired upper bound for the failure probability.)} Using the fact that $k\gg n^{1/3} (\log n)^{4/3} (\log \log n)^{1/3}$ we get $x,c\ll y$.  So Lemma~\ref{lem:cxy_process} implies that at the end of the process, with probability at least $1-\exp (- x/(20c) ) = 1 - o(n^{-1})$, we have
\begin{eqnarray*}
X &=& O(xt_{\ell}/y)=O(xn/y) = O \left( \frac {n (\log n)^2 / k}{(nk/\log(k/n^{1/3}))^{1/2}} n \right) \\
&=& O \left( \left( \frac {n (\log n)^4 \log(k/n^{1/3})}{k^3} \right)^{1/2} n \right)=o(n).
\end{eqnarray*}
Multiplying by the $k$ choices for $v$, with probability at least $1-o(1)$, at the end of the process every special component has size $o(n)$.  This completes the proof of Theorem~\ref{thm:main}(b).

\subsection{Extending the argument for large values of $k$}\label{sec:part_c}

Until now, we have assumed that $k\ll n/\log n$.  To extend to higher values of $k$, we apply Observation~\ref{obiggerk}.  For any $\omega\rightarrow\infty$ with $n$, set $k'=n/(\omega\log n)$.  Our bound above yields that a.a.s. the largest special component at the end of the process has size $O(\omega^{3/2} \log^{4} n)$.  Observation~\ref{obiggerk} says that the same bound holds for all $k\geq k'$, thus proving Theorem~\ref{thm:main}(c).

\section{Concluding Remarks}\label{sec:remarks}

Note that in Theorem~\ref{thm:main}(b) we needed to assume that $k \gg n^{1/3} (\log n)^{4/3} (\log \log n)^{1/3}$. This seems to be an artifact of the proof technique we use (the union bound over all special components) rather than the lower bound that is needed. It is natural to conjecture that a.a.s.\ $L_1(\Pk(n,M))=o(n)$ even for $k \gg n^{1/3}$. Indeed, if $k=n^{1/3} \omega$ for any $\omega = \omega(n) \to \infty$, one can show (for example, using the argument as in the proof of Lemma~\ref{lem:prop_m3}(b)) that in $\Gk(n,n/2)$, a.a.s.\ the total size of all special components is of order $n^{2/3} \sqrt{\omega}$. From the observations in Section~\ref{sec:known-results} we know that a.a.s.\ the giant component has size at most $n^{2/3} \omega^{1/4}$ (in fact, of order $n^{2/3}$; as usual, we make some room for the argument to work), and the largest component that appears after time $n/2$ is of order $n^{2/3}$. By Lemma~\ref{lem:cxy_process}, we get that a.a.s.\ the largest special component at time $m=n/2$ grows only to size $o(n)$. This supports the conjecture but it is not clear how to avoid using the union bound and so it remains an open problem.


\bigskip

Finally, we would like to thank Megan Dewar and John Proos from the Tutte Institute for Mathematics and Computing for stimulating discussions on the problem and its applications.

\bibliographystyle{apalike}
\bibliography{k-process}

\begin{thebibliography}{}

\bibitem[Aldous, 1990]{A_cycle-free}
Aldous, D. (1990).
\newblock A random tree model associated with random graphs.
\newblock {\em Random Structures and Algorithms}, 1:383--402.

\bibitem[Bohman, 2009]{Bohman_Triangles}
Bohman, T. (2009).
\newblock The triangle-free process.
\newblock {\em Advances in Mathematics}, 221:1653--1677.

\bibitem[Bollob{\'a}s, 2001]{bol}
Bollob{\'a}s, B. (2001).
\newblock {\em Random Graphs}.
\newblock Cambridge University Press.

\bibitem[Bollob\'as and Riordan, 2000]{BR_H-free}
Bollob\'as, B. and Riordan, O. (2000).
\newblock Constrained graph processes.
\newblock {\em Electronic Journal of Combinatorics}, 7(1):R18.

\bibitem[Darling et~al., 2000]{dhpp}
Darling, Harris, Phulara, and Proos (2000).
\newblock The combinatorial data fusion problem: Graph cut problems for big
  data.
\newblock {\em Conference presentation}.

\bibitem[Erd\H{o}s and R\'enyi, 1959]{ER1959}
Erd\H{o}s, P. and R\'enyi, A. (1959).
\newblock On random graphs. i.
\newblock {\em Publicationes Mathematicae}, 6:290--297.

\bibitem[Erd\H{o}s et~al., 1995]{ESW_odd-cycle-free}
Erd\H{o}s, P., Suen, S., and Winkler, P. (1995).
\newblock On the size of a random maximal graph.
\newblock {\em Random Structures and Algorithms}, 6:309--318.

\bibitem[Frieze and Karo{\'n}ski, 2015]{BookFK}
Frieze, A. and Karo{\'n}ski, M. (2015).
\newblock {\em Introduction to random graphs}.
\newblock Cambridge University Press.

\bibitem[Garey and Johnson, 1979]{gj}
Garey, M. and Johnson, D. (1979).
\newblock {\em Computers and Intractability: A Guide to the Theory of
  NP-Completeness}.
\newblock W.H.\ Freeman and Company.

\bibitem[Gerke et~al., 2008]{RSA:RSA20186}
Gerke, S., Schlatter, D., Steger, A., and Taraz, A. (2008).
\newblock The random planar graph process.
\newblock {\em Random Structures \& Algorithms}, 32(2):236--261.

\bibitem[Janson and Luczak, 2008]{jansonluczak2008}
Janson, S. and Luczak, M. (2008).
\newblock Susceptibility in subcritical random graphs.
\newblock {\em Journal of Mathematical Physics}, 49(12):125207.

\bibitem[Janson et~al., 2000]{JLR}
Janson, S., \L{}uczak, T., and Ruci\'nski, A. (2000).
\newblock {\em Random graphs}.
\newblock John Wiley \& Sons.

\bibitem[{\L}uczak, 1990]{luczak1990}
{\L}uczak, T. (1990).
\newblock Component behavior near the critical point of the random graph
  process.
\newblock {\em Random Structures and Algorithms}, 1:287--310.

\bibitem[McDiarmid, 1998]{mcdiarmid1998}
McDiarmid, C. (1998).
\newblock {\em Concentration}, pages 195--248.
\newblock Springer Berlin.

\bibitem[Osthus and Taraz, 2001]{OT_H-free}
Osthus, D. and Taraz, A. (2001).
\newblock Random maximal $h$-free graphs.
\newblock {\em Random Structures and Algorithms}, 18:61--82.

\bibitem[Ruci\'nski and Wormald, 1992]{RW_bounded_degree}
Ruci\'nski, A. and Wormald, N. (1992).
\newblock Random graph processes with degree restrictions.
\newblock {\em Combinatorics, Probability and Computing}, 221:169--180.

\bibitem[Vazirani, 2003]{vazirani}
Vazirani, V. (2003).
\newblock {\em Approximation Algorithms}.
\newblock Springer-Verlag, Berlin and Heidelberg.

\end{thebibliography}

\end{document}